\documentclass[reqno]{amsart}

\usepackage{tabu}
\usepackage{amssymb}
\usepackage{mathtools}
\usepackage{a4wide,amsmath}
\usepackage{mathrsfs}
\usepackage{amsthm}
\numberwithin{equation}{section}
\numberwithin{figure}{section}
\numberwithin{table}{section}
\usepackage{bbm}
\usepackage{subfig}
\usepackage{enumerate}
\usepackage[section]{placeins}
\usepackage{graphicx}		  
\usepackage{ifpdf}
\ifpdf
\DeclareGraphicsExtensions{.pdf,.eps,.jpg,.png}	
\usepackage[suffix=]{epstopdf}
\fi
\usepackage{xcolor}
\usepackage[utf8]{inputenc}
\usepackage{hyperref}
\hypersetup{hidelinks}

\long\def\MSC#1\EndMSC{\def\arg{#1}\ifx\arg\empty\relax\else
	{\narrower\noindent%
		{2020 Mathematics Subject Classification}: #1\\} \fi}
\long\def\PACS#1\EndPACS{\def\arg{#1}\ifx\arg\empty\relax\else
	{\narrower\noindent%
		{PACS numbers}: #1}\fi}
\long\def\KEY#1\EndKEY{\def\arg{#1}\ifx\arg\empty\relax\else
	{\narrower\noindent%
		Keywords: #1\\}\fi}

\theoremstyle{plain}
\newtheorem{theorem}{Theorem}[section]
\newtheorem{lemma}[theorem]{Lemma}
\newtheorem{proposition}[theorem]{Proposition}

\theoremstyle{definition}
\newtheorem{definition}[theorem]{Definition}

\theoremstyle{remark}
\newtheorem{remark}[theorem]{Remark}

\newcommand{\norm}[1]{\lVert#1\rVert}
\newcommand{\abs}[1]{\lvert#1\rvert} 
\newcommand{\inner}[1]{\langle#1\rangle}

\newcommand{\dist}{\mathop{\textup{dist}}}
\newcommand{\essinf}{\mathop{\textup{ess\,inf}}}
\newcommand{\esssup}{\mathop{\textup{ess\,sup}}}

\newcommand{\R}{\mathbb{R}}
\newcommand{\N}{\mathbb{N}}

\newcommand{\di}{\mathrm{d}}   
\newcommand{\A}{\mathcal{A}}
\newcommand{\X}{\mathcal{X}}
\newcommand{\Hs}{\mathcal{H}}

\begin{document}
	\title[Reconstruction of cracks in Calder\'on's problem]{Reconstruction of cracks in Calder\'on's inverse conductivity problem using energy comparisons}
	
	\author[H.~Garde]{Henrik~Garde}
	\address[H.~Garde]{Department of Mathematics, Aarhus University, Aarhus, Denmark.}
	\email{garde@math.au.dk}
	
	\author[M.~S.~Vogelius]{Michael~S.~Vogelius}
	\address[M.~S.~Vogelius]{Department of Mathematics, Rutgers University, New Brunswick, NJ, USA.}
	\email{vogelius@math.rutgers.edu}
	
	\begin{abstract}
		We derive exact reconstruction methods for cracks consisting of unions of Lipschitz hypersurfaces in the context of Calder\'on's inverse conductivity problem. Our first method obtains upper bounds for the unknown cracks, bounds that can be shrunk to obtain the exact crack locations upon verifying certain operator inequalities for differences of the local Neumann-to-Dirichlet maps. This method can simultaneously handle perfectly insulating and perfectly conducting cracks, and it appears to be the first rigorous reconstruction method capable of this. Our second method assumes that only perfectly insulating cracks or only perfectly conducting cracks are present. Once more using operator inequalities, this method generates approximate cracks that are guaranteed to be subsets of the unknown cracks that are being reconstructed. 
	\end{abstract}	
	\maketitle
	
	\KEY
	Calder\'on's problem,
	electrical impedance tomography, 
	cracks,
	inclusion detection.
	\EndKEY
	
	\MSC
	35R30, 35R05, 47H05.
	\EndMSC
	
	\section{Introduction} \label{sec:intro}
	
	For an integer $d\geq 2$, let $\Omega$ be a bounded Lipschitz domain in $\R^d$ with connected complement, and let $\Gamma\subseteq \partial \Omega$ be a non-empty relatively open subset. In the presence of a collection of cracks, $D$,  and an imposed boundary current, $f$ on $\Gamma$, the steady state voltage potential, $u$, satisfies
	\begin{align} \label{eq:condeq}
		-\nabla\cdot(\gamma_0\nabla u) &= 0 \quad \text{in } \Omega\setminus \overline{D}, \nonumber\\
		\gamma_0\frac{\partial u}{\partial \nu} &= \begin{dcases}
			f & \text{ on } \Gamma, \\
			0 & \text{ on } \partial\Omega\setminus\overline{\Gamma},
		\end{dcases} \nonumber\\
		\gamma_0\frac{\partial u}{\partial n} &= 0 \quad \text{on } D_0, \\[1mm]
		& \hspace{-1.5cm} u \text{ is locally constant on } D_\infty, \nonumber \\
	\int_{D_i} \Bigl[\gamma_0\frac{\partial u}{\partial n}\Bigr] \di S &= 0 \quad \text{for each component $D_i$ of $D_\infty$}. \nonumber
	\end{align}	
	We ``ground" $u$ by requiring that $\int_\Gamma u\,\di S = 0$. The condition on $D_0$ must be satisfied on both sides of $D_0$. $D_0$ are those cracks in the collection $D$ that are perfectly insulating, and $D_\infty$ are those that are perfectly conducting, $D=D_0\cup D_\infty$. Note that the constants on different connected components of $D_\infty$ will generally be different (and determined by the integral conditions).	$\nu$ denotes the outer unit normal on $\partial\Omega$ and $n$ denotes a unit normal on $D$. We use $[\, v \,]$ to denote the jump of a function $v$ across $D$, relative to the normal vector $n$, with the convention that
$[\, v\,]= v_--v_+$, where $v_-$ is the trace (on $D$) taken from that part of $\Omega$, which $n$ points away from, and $v_+$ is the trace (on $D$) taken from that part of $\Omega$, which $n$ points into. The imposed boundary current satisfies
	\begin{equation} \label{eq:currents}
		f \in L^2_\diamond(\Gamma) = \{ g\in L^2(\Gamma) \mid \inner{1,g} = 0 \},
	\end{equation}
	where $\inner{\,\cdot\,,\,\cdot\,}$ is the usual $L^2(\Gamma)$ inner product. $\gamma_0$ is a known background conductivity coefficient, which is defined in all of $\Omega$. 
	
	For the major part of our analysis it suffices that $\gamma_0\in L^\infty(\Omega)$ with $\essinf(\gamma_0)>0$. However, at certain points we crucially assume that the conductivity equation $\nabla \cdot(\gamma_0 \nabla v)=0$ has a weak unique continuation property (UCP) on connected open subsets of $\Omega$ and also from Cauchy data on $\Gamma$, see e.g.~\cite[Definition~3.3]{Garde2020}. For $d=2$ the UCP follows directly \cite{Alessandrini2012}, while in higher dimensions this requires some regularity of $\gamma_0$ such as $\gamma_0$ being Lipschitz continuous or piecewise analytic \cite{Jerison85,KohnVogelius85,Miller74}. At one particular point (Lemma \ref{lemma:nontrivial}) we require additional $C^2$ regularity, namely to ensure that the given cracks cannot be invisible to all possible boundary measurements. However, if it turns out that the visibility can be shown for a lower regularity (still satisfying the UCP) then the proofs of our main results adapt to this.
	
	Formally the problem introduced in \eqref{eq:condeq} can be considered a problem related to the conductivity profile given by
	\begin{equation*}
		\begin{dcases}
			\gamma_0  &\text{ in } \Omega\setminus D, \\
			0 &\text{ in } D_0, \\
			\infty &\text{ in } D_\infty,
		\end{dcases}
	\end{equation*}
which also explains our use of the notation $D_0$ and $D_\infty$. The problem of determining an internal conductivity profile from boundary data is often referred to as Calder\'on's inverse conductivity problem \cite{Calderon1980,Uhlmann2009}, and our emphasis here is therefore on a special case of this, namely: to reconstruct the unknown cracks in $\Omega$ (the set $D$) based on boundary measurements on $\Gamma \subseteq \partial \Omega$. 

In order to state the main results of this paper, it is necessary to be precise about the collection of cracks we reconstruct.

\begin{definition}{}\ \label{def:crack}
 A collection of cracks $\chi$ lies in the class $\X$  if for some $N\in\N_0$
			\begin{equation*}
				\chi = \bigcup_{i=1}^N \sigma_i,
			\end{equation*} 
			where the $\sigma_i\subset \Omega$ are $(d-1)$ dimensional connected orientable Lipschitz surfaces with non-empty Lipschitz boundary $\partial \sigma_i$, and with 
			\begin{equation*}
				\dist(\sigma_i,\sigma_j)>0  \text{ for } i\neq j \text{ and } \dist(\sigma_i,\partial \Omega)>0 \text{ for all } i.
			\end{equation*}
			%
	\end{definition}
\vskip 7pt	
	Notice that $\dist(\chi,\partial\Omega) > 0$ and $\Omega\setminus\overline{\chi}$ is connected. In particular, we do not allow open subsets of $\Omega$ to be completely encapsulated by a crack. We refer to  $D \in \X$ as ``a $(D_0,D_\infty)$ collection of cracks'' if $D = D_0\cup D_\infty$ for $D_0,D_\infty\in\X$, $\dist(D_0,D_\infty)>0$, each crack in $D_0$ is perfectly insulating, and each crack in $D_\infty$ is perfectly conducting. Note that $D_0$ or $D_\infty$ may possibly be empty.

	Our first result is an exact monotonicity-based reconstruction method via shape approximations from above, based on shrinking test inclusions $C$ containing $D$. Whether $C$ contains $D$ or not is checked via operator inequalities related to the boundary electric power. 
	The set of admissible test inclusions are
	\begin{align*}
		\A &= \{ C \Subset \Omega \mid C \text{ is the closure of an open set, has connected }  \\
		&\hphantom{{}= \{C \subset \Omega \mid{}}\text{complement, and has Lipschitz boundary } \partial C\}.
	\end{align*}
	For $C\in\A$ we define $\Lambda_{C}^{\emptyset}$ to be the Neumann-to-Dirichlet map (the ND map) for the conductivity coefficient which equals $0$ in $C$ and $\gamma_0$ elsewhere. We define $\Lambda_{\emptyset}^C$ to be the ND map for the conductivity coefficient which equals $\infty$ in $C$ and $\gamma_0$ elsewhere. Similarly we define $\Lambda_{D_0}^{D_\infty}$ to be the ND map associated with the $(D_0,D_\infty)$ collection of cracks. For the precise definitions see section~\ref{sec:forward} or e.g.~\cite{Garde2020}. We shall actually only be concerned with a local version of these maps from $L^2_\diamond(\Gamma)$ to $L^2_\diamond(\Gamma)$. For self-adjoint operators $A,B\in\mathscr{L}(L^2_\diamond(\Gamma))$ we use the notation  $A\geq B$ if and only if $A-B$ is positive semi-definite, that is
	\begin{equation*}
		\inner{(A-B)f,f}\geq 0, \qquad \forall f\in L^2_\diamond(\Gamma).
	\end{equation*}
	Based on knowledge of the local ND map $\Lambda_{D_0}^{D_\infty} : f\mapsto u|_\Gamma$, and in particular its relation to $\Lambda_{C}^{\emptyset}$  and $\Lambda_{\emptyset}^{C}$, we can now reconstruct the cracks $D$ as follows.
	\begin{theorem} \label{thm:main}
		Let $D$ be a $(D_0,D_\infty)$ collection of cracks. Given any $C\in\A$, then
		\begin{equation*}
			D\subset C \qquad \text{if and only if} \qquad \Lambda_{C}^{\emptyset}\geq \Lambda_{D_0}^{D_\infty} \geq \Lambda_{\emptyset}^C.
		\end{equation*}
	\end{theorem}
	\begin{proof}
		The direction ``$\Rightarrow$'' is proved in Proposition~\ref{prop:main1} in section~\ref{sec:proof_main1} with $\Sigma_0 = D_0$ and $\Sigma_\infty = D_\infty$. The other direction ``$\Leftarrow$'' is proved in Proposition~\ref{prop:main2} in section~\ref{sec:proof_main2}.
	\end{proof}
	
	Our second result provides an exact monotonicity-based reconstruction method based on generating approximate cracks $\chi$ inside $D$. By checking operator inequalities we can determine if $\chi$ is contained in $D$ or not. This result assumes that we either only have perfectly insulating cracks or only have perfectly conducting cracks.
	\begin{theorem} \label{thm:main2}
		Let $D\in \X$.
		\begin{enumerate}[\rm(i)]
			\item Given any $\chi\in\X$, then
			\begin{equation*}
				\chi\subseteq D \qquad \text{if and only if} \qquad \Lambda_{D}^{\emptyset} \geq \Lambda_{\chi}^{\emptyset}.
			\end{equation*}
			\item Given any $\chi\in\X$, then
			\begin{equation*}
				\chi\subseteq D \qquad \text{if and only if} \qquad \Lambda_{\emptyset}^{\chi} \geq \Lambda_{\emptyset}^{D}.
			\end{equation*}
		\end{enumerate}
	\end{theorem}
	\begin{proof}
		The direction ``$\Rightarrow$'' of (i) is proved in Proposition~\ref{prop:main12}(i) in section~\ref{sec:proof_main1} with $\Sigma_0 = D$ and $\Sigma_\infty = \emptyset$, and the same direction for (ii) is proved in Proposition~\ref{prop:main12}(ii) with $\Sigma_0 = \emptyset$ and $\Sigma_\infty = D$. The other direction ``$\Leftarrow$'' for both (i) and (ii) is proved in Proposition~\ref{prop:main22} in section~\ref{sec:proof_main2}.
	\end{proof}
	
	\begin{remark}{}\
		\begin{enumerate}[(i)]
			\item Theorem~\ref{thm:main} implies that 
			\begin{equation*}
				\overline{D} = \cap \{ C\in\A \mid \Lambda_{C}^{\emptyset}\geq \Lambda_{D_0}^{D_\infty} \geq \Lambda_{\emptyset}^C\}. 
			\end{equation*}
			Moreover, as seen in the proof of Proposition~\ref{prop:main2}, if $D_0 = \emptyset$ we only need to consider $\Lambda_{\emptyset}^{D_\infty} \geq \Lambda_{\emptyset}^C$ and, likewise, if instead we have $D_\infty = \emptyset$ we only need to consider $\Lambda_{C}^{\emptyset}\geq \Lambda_{D_0}^{\emptyset}$.
			\item Theorem~\ref{thm:main2} implies that, in the perfectly insulating case
			\begin{equation*}
				D = \cup \{\chi\in\X\mid \Lambda_{D}^{\emptyset} \geq \Lambda_{\chi}^{\emptyset}\},
			\end{equation*}
			and in the perfectly conducting case
			\begin{equation*}
				D = \cup \{\chi\in\X\mid \Lambda_{\emptyset}^{\chi} \geq \Lambda_{\emptyset}^{D}\}.
			\end{equation*}
		\end{enumerate}
	\end{remark}

	\subsection{Some related results}
	
	For large classes of cracks there are optimal (non-constructive) unique identifiability results using merely two boundary measurements \cite{Alessandrini97,BryanVogelius92,FriedmanVogelius89}. However, the \emph{reconstruction of general cracks} from finitely many measurements appears to be an open problem. For spatial dimension $d=2$ there is a factorisation method \cite{BruhlHankePidcock2001} to reconstruct \emph{either} perfectly insulating cracks \emph{or} perfectly conducting cracks from an ND map. The results of \cite{BruhlHankePidcock2001} resemble the two-dimensional version of Theorem~\ref{thm:main2}, in the sense that one determines if an artificial crack is a subset of the unknown crack. For perfectly insulating cracks, probe and enclosure methods are discussed in \cite{Ikehata08, Ikehata2023}. These methods determine the location of  ``boundary singularities" associated with the cracks or the convex hull of the cracks, respectively. It should also be mentioned that these reconstruction methods assume that $\gamma_0\equiv 1$.
	
	Our results Theorem~\ref{thm:main} and Theorem~\ref{thm:main2} are applicable in spatial dimension $d\geq 2$ for very general cracks and also for non-homogeneous background conductivities. It appears that Theorem~\ref{thm:main} is the only proven method capable of reconstructing cracks with both perfectly insulating and perfectly conducting parts. Moreover, the test operators $\Lambda_C^\emptyset$ and $\Lambda_\emptyset^C$ are precisely the same test operators that are used in \cite{Garde2020,Garde2022b} for reconstructing general inclusions of positive volume (the support of conductivity perturbations on open sets). So numerical implementations of this method directly applies to cracks without requiring any modifications. Arguably Theorem~\ref{thm:main2} becomes demanding to implement numerically for $d>2$ and is likely only suitable for computations in $d=2$, while the numerical implementation of Theorem~\ref{thm:main} very naturally generalises to higher dimensions using a peeling approach as in \cite{GardeStaboulis2019}. 
	
	For inclusions of positive volume, the monotonicity of ND mappings with respect to the conductivity coefficient was used in \cite{Tamburrino2002} to give bounds on the inclusions. In \cite{Harrach10,Harrach13} this approach was proven to give an exact reconstruction method for finite perturbations to the background conductivity, also in cases with both positive and negative perturbations if lower and upper bounds are known for the perturbed conductivity. This was generalised in \cite{Garde2020}, where the perturbed conductivity could now simultaneously have parts with finite positive and negative perturbations as well as extreme parts that are perfectly insulating and perfectly conducting. Moreover, the need for lower and upper bounds was removed. Finally, in \cite{Garde2022b} the method was shown to also handle degenerate and singular perturbations based on $A_2$-Muckenhoupt weights, allowing for continuous decay to zero and growth to infinity. This is currently the most general method for reconstructing inclusions of positive volume in Calder\'on's problem based on a local ND map. Rigorous connections have also been made to practically relevant electrode models \cite{GardeStaboulis2016,GardeStaboulis2019,Harrach19,Harrach15}. 
	
	The same monotonicity-based methodology has also been used to design a reconstruction method for the partial data Calder\'on problem, for the case of layered piecewise constant conductivities \cite{Garde2020b,Garde2022}; note that this is not a finite-dimensional setting, as the piecewise constant partitioning is also reconstructed. In finite-dimensional settings this methodology has led to Lipschitz stability results in Calder\'on's problem with finitely many measurements \cite{Harrach19} and to reformulating the finite-dimensional Calder\'on problem as a convex semidefinite optimisation problem \cite{Harrach23}. The methods are based on the unique continuation principle and its connection to the localised potentials from \cite{Gebauer2008b}. For applications to other inverse problems, we refer to the list of references in \cite{Harrach19}.
	
	The proofs of Theorems~\ref{thm:main}~and~\ref{thm:main2} also rely on localised potentials. However, the approach here is quite different since the known localisation results apply to open sets, and there is no open set inside the cracks to localise in. Moreover, the monotonicity inequalities in \cite[Appendix~A]{Garde2020} become trivial in the limit of approximating cracks by open sets. Additionally, it should be noted that using approximations of cracks with open sets and applying the results from \cite{Garde2020} will not result in proofs of our main results. The key ingredient for our proofs turns out to be the localisation on open sets containing parts of the cracks, to give simultaneous blow-up of potentials for conductivity profiles both with and without the cracks, and crucially also to have blow-up for their difference. The latter is the most technically difficult to obtain, and here the constructive version of localised potentials \cite[Lemma~2.8]{Gebauer2008b} turns out to be invaluable, since the existence results for localised potentials would not necessarily give blow-up for the difference of two localised sequences of potentials. 
	
	\subsection{Outline of the paper}
	
	In section~\ref{sec:forward} we introduce the relevant forward problems and associated function spaces. In section~\ref{sec:proof_main1} we prove the direction ``$\Rightarrow$'' of Theorems~\ref{thm:main}~and~\ref{thm:main2} (the easy direction). Sections~\ref{sec:locpot1}~and~\ref{sec:locpot2} are dedicated to results about localised potentials, that are needed in section~\ref{sec:proof_main2}, where we prove the direction ``$\Leftarrow$'' of Theorems~\ref{thm:main}~and~\ref{thm:main2} (the difficult direction). 
	
	\section{Forward problems} \label{sec:forward}
	
	Let $\Sigma$ be a $(\Sigma_0,\Sigma_\infty)$ collection of cracks. Let 
	\begin{equation*}
		\Hs_{\Sigma_0}^{\Sigma_\infty} = \{ v\in H_\diamond^1(\Omega\setminus \overline{\Sigma_0}) \mid \text{$v$ is locally constant on $\Sigma_\infty$} \},
	\end{equation*}
	where the $\diamond$-symbol refers to a mean-free condition for the Dirichlet trace on $\Gamma$, as in \eqref{eq:currents}. A Poincar\'e inequality holds on $\Omega\setminus \overline{\Sigma_0}$, see for instance~\cite{Lieb2003}. As $\Sigma_0$ has $d$-Lebesgue measure zero, the gradient of $v\in\Hs_{\Sigma_0}^{\Sigma_\infty}$ (defined in $\Omega\setminus \overline{\Sigma_0}$) is equivalent to a function in $L^2(\Omega)^d$ on the entire domain $\Omega$. Hence a norm on $\Hs_{\Sigma_0}^{\Sigma_\infty}$, equivalent to the $H^1$-norm, is 
	\begin{equation*}
		\norm{v}_{*} = \Bigl(\int_{\Omega} \gamma_0\abs{\nabla v}^2\,\di x\Bigr)^{1/2},
	\end{equation*}
	with associated inner product\footnote{Throughout this paper we assume, without loss of generality, that all our solutions and function spaces are real-valued.}
	\begin{equation*}
		\inner{w,v}_{*} = \int_{\Omega} \gamma_0\nabla w\cdot\nabla v\,\di x. 
	\end{equation*}
	By the Lax--Milgram lemma, the interior electric potential $u$ in \eqref{eq:condeq}, with $D_0 = \Sigma_0$ and $D_\infty = \Sigma_\infty$, is the unique solution in $\Hs_{\Sigma_0}^{\Sigma_\infty}$ of the weak problem
	\begin{equation} \label{eq:weakprob}
		\inner{u,v}_{*} = \inner{f,v|_\Gamma}, \qquad \forall v\in\Hs_{\Sigma_0}^{\Sigma_\infty}.
	\end{equation}
	We will sometimes use the notation $u = u_{\Sigma_0,f}^{\Sigma_\infty}$ to clarify which cracks (possibly empty) are present, and which current density, $f$, is used. Define the functional $J : \Hs_{\Sigma_0}^{\Sigma_\infty} \to \R$ by
	\begin{equation*}
		J(v) = \norm{v}_*^2 - 2 \inner{f,v|_\Gamma}, \qquad v\in \Hs_{\Sigma_0}^{\Sigma_\infty}.
	\end{equation*}
	With this definition,  $u$ is also the unique minimiser of $J$ (cf. ~\cite[Remark~12.23]{Grubb} and~\cite[Theorem~1.1.2]{Ciarlet1978}). The corresponding local ND map $\Lambda_{\Sigma_0}^{\Sigma_\infty} : f\mapsto u|_\Gamma$ satisfies
	\begin{equation*}
		\inner{\Lambda_{\Sigma_0}^{\Sigma_\infty}f,f} = \inner{u|_\Gamma,f} = \inner{f,u|_\Gamma} = \norm{u}_*^2,
	\end{equation*}
	and is a compact self-adjoint operator on $L^2_\diamond(\Gamma)$.

	For $C\in\A$ and $f\in L^2_\diamond(\Gamma)$, let $u_{C,f}^\emptyset$ be the potential coming from a conductivity with perfectly insulating inclusions in $C$ and $\gamma_0$ elsewhere. Then $u_0 = u_{C,f}^\emptyset$ is the unique solution in 
	\begin{equation*}
		\Hs_C^\emptyset = H_\diamond^1(\Omega\setminus C)
	\end{equation*}
	of the variational problem
	\begin{equation*}
		\int_{\Omega\setminus C} \gamma_0 \nabla u_0\cdot\nabla v \,\di x = \inner{f,v|_\Gamma}, \qquad \forall v\in \Hs_C^\emptyset.
	\end{equation*}
	Likewise, let $u_{\emptyset,f}^C$ be the potential coming from a conductivity with perfectly conducting inclusions in $C$ and $\gamma_0$ elsewhere. Then $u_\infty = u_{\emptyset,f}^C$ is the unique solution in
	\begin{equation*}
		\Hs_\emptyset^C = \{w\in H_\diamond^1(\Omega) \mid w \text{ is locally constant in } C\}
	\end{equation*} 
	of the variational problem 
	\begin{equation*}
		\int_{\Omega\setminus C} \gamma_0 \nabla u_\infty\cdot \nabla v \,\di x = \inner{f,v|_\Gamma}, \qquad \forall v\in \Hs_\emptyset^C.
	\end{equation*}
	In terms of  the functional
	\begin{equation*}
		\widetilde{J}(v) = \int_{\Omega\setminus C} \gamma_0 \abs{\nabla v}^2\,\di x - 2 \inner{f,v|_\Gamma},
	\end{equation*}
$u_0$ is the unique minimizer of $\widetilde{J}|_{\Hs_C^\emptyset}$ and $u_\infty$ is the unique minimizer of $\widetilde{J}|_{\Hs_\emptyset^C}$. The corresponding ND maps are compact self-adjoint operators on $L^2_\diamond(\Gamma)$ and satisfy
	\begin{align*}
		\inner{\Lambda_{C}^{\emptyset} f,f} &= \inner{u_0|_\Gamma,f} = \inner{f,u_0|_\Gamma} = \int_{\Omega\setminus C} \gamma_0\abs{\nabla u_0}^2\,\di x,\\
		\inner{\Lambda_{\emptyset}^C f,f} &= \inner{u_\infty|_\Gamma,f} = \inner{f,u_\infty|_\Gamma} = \int_{\Omega\setminus C} \gamma_0\abs{\nabla u_\infty}^2\,\di x.
	\end{align*}
	
	\section{The direction ``$\Rightarrow$'' in Theorems~\ref{thm:main} \& \ref{thm:main2}} \label{sec:proof_main1}
	
	Below we prove the ``easy direction'' of the \emph{if and only if} statements of Theorems~\ref{thm:main} and \ref{thm:main2}.
	
	\begin{proposition} \label{prop:main1}
		Let $\Sigma$ be a $(\Sigma_0,\Sigma_\infty)$ collection of cracks. Given any $C\in\A$, then
		\begin{equation*}
			\Sigma\subset C \qquad \text{implies} \qquad \Lambda_{C}^{\emptyset}\geq \Lambda_{\Sigma_0}^{\Sigma_\infty} \geq \Lambda_{\emptyset}^C.
		\end{equation*}
	\end{proposition}
	\begin{proof}		
		For $f\in L_\diamond^2(\Gamma)$ denote $u_0 = u_{C,f}^\emptyset$, $u_\infty = u_{\emptyset,f}^C$, and $u = u_{\Sigma_0,f}^{\Sigma_\infty}$. Since $\Sigma\subset C$ we have $u|_{\Omega\setminus C}\in\Hs_C^\emptyset$ and $u_\infty|_{\Omega\setminus \overline{\Sigma_0}} \in \Hs_{\Sigma_0}^{\Sigma_\infty}$. We now compare the ND maps, while taking the related minimization properties from section~\ref{sec:forward} into account
		\begin{align*}
			-\inner{\Lambda_{\Sigma_0}^{\Sigma_\infty}f,f} &= \int_{\Omega} \gamma_0\abs{\nabla u}^2\,\di x - 2\inner{f,u|_\Gamma} \\
			&\geq \int_{\Omega\setminus C} \gamma_0\abs{\nabla u}^2\,\di x - 2\inner{f,u|_\Gamma} \\
			&\geq \int_{\Omega\setminus C} \gamma_0\abs{\nabla u_0}^2\,\di x - 2\inner{f,u_0|_\Gamma} \\
			&= -\inner{\Lambda_C^\emptyset f,f},
		\end{align*}
		and
		\begin{align*}
			-\inner{\Lambda_\emptyset^C f,f} &= \int_{\Omega\setminus C} \gamma_0\abs{\nabla u_\infty}^2\,\di x - 2\inner{f,u_\infty|_\Gamma} \\
			&= \int_{\Omega} \gamma_0\abs{\nabla u_\infty}^2\,\di x - 2\inner{f,u_\infty|_\Gamma} \\
			&\geq \int_{\Omega} \gamma_0\abs{\nabla u}^2\,\di x - 2\inner{f,u|_\Gamma} \\
			&= -\inner{\Lambda_{\Sigma_0}^{\Sigma_\infty} f,f}.
		\end{align*}
		Here we used that $\nabla u_\infty \equiv 0$ in the interior of $C$.
	\end{proof}

	\begin{proposition} \label{prop:main12}
		Let $\Sigma$ be a $(\Sigma_0,\Sigma_\infty)$ collection of cracks.
		\begin{enumerate}[\rm(i)]
			\item Given any $\chi\in\X$, then
			\begin{equation*}
				\chi\subseteq \Sigma_0 \qquad \text{implies} \qquad \Lambda_{\Sigma_0}^{\Sigma_\infty} \geq \Lambda_{\chi}^{\Sigma_\infty}.
			\end{equation*}
			\item Given any $\chi\in\X$, then
			\begin{equation*}
				\chi\subseteq \Sigma_\infty \qquad \text{implies} \qquad \Lambda_{\Sigma_0}^{\chi} \geq \Lambda_{\Sigma_0}^{\Sigma_\infty}.
			\end{equation*}
		\end{enumerate}
	\end{proposition}
	\begin{proof}		
		For $f\in L_\diamond^2(\Gamma)$ denote $u = u_{\Sigma_0,f}^{\Sigma_\infty}$. 
		
		Proof of (i): We denote $u_1 = u_{\chi,f}^{\Sigma_\infty}$. Since $\chi\subseteq \Sigma_0$ we have $u_1|_{\Omega\setminus \overline{\Sigma_0}}\in\Hs_{\Sigma_0}^{\Sigma_\infty}$. Using the related minimization problems for $u$ and $u_1$, we have
		\begin{equation*}
			-\inner{\Lambda_{\chi}^{\Sigma_\infty}f,f} = \norm{u_1}_*^2 - 2\inner{f,u_1|_\Gamma} \geq \norm{u}_*^2 - 2\inner{f,u|_\Gamma} = -\inner{\Lambda_{\Sigma_0}^{\Sigma_\infty} f,f}.
		\end{equation*}
		
		Proof of (ii):  We denote $u_2 = u_{\Sigma_0,f}^{\chi}$. Since $\chi\subseteq \Sigma_\infty$ we have $u\in\Hs_{\Sigma_0}^{\chi}$. Using the related minimization problems for $u$ and $u_2$, we have
		\begin{equation*}
			-\inner{\Lambda_{\Sigma_0}^{\Sigma_\infty}f,f} = \norm{u}_*^2 - 2\inner{f,u|_\Gamma} \geq \norm{u_2}_*^2 - 2\inner{f,u_2|_\Gamma} = -\inner{\Lambda_{\Sigma_0}^{\chi} f,f}. \qedhere
		\end{equation*}
	\end{proof}

	\section{Some lemmas in preparation for localised potentials} \label{sec:locpot1}
	
	If $\Sigma$ is a $(\Sigma_0,\Sigma_\infty)$ collection of cracks  then we have 
	\begin{equation*}
		\Hs_\emptyset^{\Sigma_\infty} \subseteq \Hs_{\Sigma_0}^{\Sigma_\infty} \subseteq \Hs_{\Sigma_0}^{\emptyset},
	\end{equation*}
	in the sense that $v\in \Hs_\emptyset^{\Sigma_\infty}$ satisfies $v|_{\Omega\setminus \overline{\Sigma_0}} \in \Hs_{\Sigma_0}^{\Sigma_\infty}$. Moreover, these are all Hilbert spaces with the same inner product $\inner{\,\cdot\,,\,\cdot\,}_*$.
	
	Let $P$ and $P^\perp$ be the orthogonal projections of $\Hs_{\Sigma_0}^{\Sigma_\infty}$ onto $\Hs_\emptyset^{\Sigma_\infty}$ and $(\Hs_\emptyset^{\Sigma_\infty})^\perp$, respectively, in the $\inner{\,\cdot\,,\,\cdot\,}_*$ inner product. Likewise we let $Q$ and $Q^\perp$ denote the orthogonal projections of $\Hs_{\Sigma_0}^{\emptyset}$ onto $\Hs_{\Sigma_0}^{\Sigma_\infty}$ and $(\Hs_{\Sigma_0}^{\Sigma_\infty})^\perp$, respectively, in the $\inner{\,\cdot\,,\,\cdot\,}_*$ inner product.
	
	\begin{lemma} \label{lemma:NDdiff}
		Let $\Sigma$ be a $(\Sigma_0,\Sigma_\infty)$ collection of cracks and let the projections $P$, $P^\perp$, $Q$, and $Q^\perp$ be given as above. For $f\in L^2_\diamond(\Gamma)$ we denote $u_0 = u_{\Sigma_0,f}^{\emptyset}$, $u_\infty = u_{\emptyset,f}^{\Sigma_\infty}$, and $u = u_{\Sigma_0,f}^{\Sigma_\infty}$.
		\begin{enumerate}[\rm(i)]
			\item We have $u_\infty = Pu$ and 
			\begin{equation*}
				\inner{(\Lambda_{\Sigma_0}^{\Sigma_\infty}-\Lambda_{\emptyset}^{\Sigma_\infty})f,f} = \norm{P^\perp u}_*^2.
			\end{equation*}
			\item We have $u = Qu_0$ and
			\begin{equation*}
				\inner{(\Lambda_{\Sigma_0}^{\emptyset}-\Lambda_{\Sigma_0}^{\Sigma_\infty})f,f} = \norm{Q^\perp u_0}_*^2.
			\end{equation*}
		\end{enumerate}
	\end{lemma}
	\begin{proof}
		Proof of (i): From the weak formulation for $u$ and for all $v\in\Hs_{\emptyset}^{\Sigma_\infty}\subseteq \Hs_{\Sigma_0}^{\Sigma_\infty}$, we have
		\begin{equation*}
			\inner{f,v|_{\Gamma}} = \inner{u,v}_* = \inner{Pu,v}_*.
		\end{equation*}
		However, by the weak formulation for $u_\infty$ and its unique solvability in $\Hs_{\emptyset}^{\Sigma_\infty}$, we have $u_\infty = Pu$. Thus $u-u_\infty = P^\perp u$. Using the weak formulation for $u$ again, we have
		\begin{equation*}
			\inner{(\Lambda_{\Sigma_0}^{\Sigma_\infty}-\Lambda_{\emptyset}^{\Sigma_\infty})f,f} = \inner{f,(P^\perp u)|_\Gamma} = \inner{u,P^\perp u}_* = \norm{P^\perp u}_*^2.
		\end{equation*}
		
		Proof of (ii): From the weak formulation for $u_0$ and for all $v\in \Hs_{\Sigma_0}^{\Sigma_\infty} \subseteq \Hs_{\Sigma_0}^{\emptyset}$, we have
		\begin{equation*}
			\inner{f,v|_{\Gamma}} = \inner{u_0,v}_* = \inner{Qu_0,v}_*.
		\end{equation*}
		However, by the weak formulation for $u$ and its unique solvability in $\Hs_{\Sigma_0}^{\Sigma_\infty}$, we have $u = Qu_0$. Thus $u_0-u = Q^\perp u_0$. Using the weak formulation for $u_0$ again, we have
		\begin{equation*}
			\inner{(\Lambda_{\Sigma_0}^{\emptyset}-\Lambda_{\Sigma_0}^{\Sigma_\infty})f,f} = \inner{f,(Q^\perp u_0)|_\Gamma} = \inner{u_0,Q^\perp u_0}_* = \norm{Q^\perp u_0}_*^2. \qedhere
		\end{equation*}
	\end{proof}
	
	\begin{remark}
		While we do not use this fact, we note that, for sufficiently regular $\gamma_0$ one can write the differences of the ND maps in Lemma~\ref{lemma:NDdiff} as integrals on $\Sigma_0$ and $\Sigma_\infty$ via integration by parts:
		\begin{align*}
			\inner{(\Lambda_{\Sigma_0}^{\Sigma_\infty}-\Lambda_{\emptyset}^{\Sigma_\infty})f,f} &= -\int_{\Sigma_0}\Bigl(\gamma_0\frac{\partial u_\infty}{\partial n}\Bigr){[u]}\,\di S, \\
			\inner{(\Lambda_{\Sigma_0}^{\emptyset}-\Lambda_{\Sigma_0}^{\Sigma_\infty})f,f} &= -\int_{\Sigma_\infty}\Bigl[\gamma_0\frac{\partial u}{\partial n}\Bigr]{u_0}\,\di S.
		\end{align*}
		The integrals should be understood in the appropriate weak sense.
	\end{remark}

	Next we state two general lemmas in functional analysis that are  proven in \cite{Gebauer2008b}. The first is a lemma relating the ranges of operators to bounds on their adjoints.
	
	\begin{lemma}[Lemma~2.5 in \cite{Gebauer2008b}] \label{lemma:rangenorm}
		Let $H$, $K_1$, and $K_2$ be Hilbert spaces and let $A_j\in\mathscr{L}(K_j,H)$ for $j = 1,2$. Then
		\begin{equation*}
			R(A_1) \subseteq R(A_2) \qquad \text{if and only if} \qquad \exists C>0, \forall x\in H: \norm{A_1^* x}_{K_1} \leq C \norm{A_2^* x}_{K_2}.
		\end{equation*}
	\end{lemma}

	The second lemma is the constructive version for the localised potentials from \cite{Gebauer2008b}.
	
	\begin{lemma}[Lemma 2.8 in \cite{Gebauer2008b}] \label{lemma:locgen}
		Let $H$, $K_1$, and $K_2$ be Hilbert spaces, let $A_j\in\mathscr{L}(K_j,H)$ for $j = 1,2$, and assume that $A_2^*$ is injective. Assume that there exists $y_0\in R(A_1)$ such that $y_0\not\in R(A_2)$. For $n\in\N$ we define
		\begin{equation*}
			\xi_n = \bigl(A_2A_2^* + \tfrac{1}{n}I\bigr)^{-1}y_0
		\end{equation*}
		and
		\begin{equation*}
			x_n = \frac{\xi_n}{\norm{A_2^* \xi_n}^{3/2}_{K_2}}.
		\end{equation*}
		Then
		\begin{equation*}
			\lim_{n\to\infty}\norm{A_1^* x_n}_{K_1} = \infty \qquad \text{and} \qquad \lim_{n\to\infty}\norm{A_2^* x_n}_{K_2} = 0.
		\end{equation*}
	\end{lemma}
	
	\section{Localised potentials with cracks} \label{sec:locpot2}
	
	Let $V\in \A$ and let $\Sigma$ be a $(\Sigma_0,\Sigma_\infty)$ collection of cracks. For $F\in L^2(V)^d$ we define $w = w_{\Sigma_0,F}^{\Sigma_\infty} \in \Hs_{\Sigma_0}^{\Sigma_\infty}$ as the unique solution of the following variational problem:
	\begin{equation} \label{eq:weq}
		\int_{\Omega} \gamma_0\nabla w\cdot\nabla v \,\di x = \int_{V} F\cdot\nabla v\,\di x, \qquad \forall v\in \Hs_{\Sigma_0}^{\Sigma_\infty}.
	\end{equation}
	The dependence on $V$ is not explicitly given in the notation of $w_{\Sigma_0,F}^{\Sigma_\infty}$, but is indirectly given via $F$. We now define an operator $L_{\Sigma_0}^{\Sigma_\infty}(V) : L^2(V)^d \to L_\diamond^2(\Gamma)$ as 
	\begin{equation*}
		L_{\Sigma_0}^{\Sigma_\infty}(V)F = w^{\Sigma_\infty}_{\Sigma_0,F}|_{\Gamma}.
	\end{equation*}
	
	Let $u = u_{\Sigma_0,f}^{\Sigma_\infty}$, then from the variational problems of $u$ and $w$ we have
	\begin{align*}
		\inner{f,L_{\Sigma_0}^{\Sigma_\infty}(V)F} &= \inner{f,w|_\Gamma} = \inner{u,w}_* = \inner{\nabla u|_{V},F}_{L^2(V)^d}.
	\end{align*}
	Hence we have
	\begin{equation} \label{eq:Ladj}
		(L_{\Sigma_0}^{\Sigma_\infty}(V))^*f = \nabla u_{\Sigma_0,f}^{\Sigma_\infty}|_{V}, \qquad f\in L_\diamond^2(\Gamma).
	\end{equation}
	Next we will prove results for the ranges of these variational operators, that will be essential for the localised potentials in Proposition~\ref{prop:locpot}.
	\begin{proposition} \label{prop:rangeindep}
		Let $V\in \A$ and let $\Sigma$ be a $(\Sigma_0,\Sigma_\infty)$ collection of cracks. 
		\begin{enumerate}[\rm(i)]
			\item If $\Sigma_0\Subset V$ then $R(L_{\emptyset}^{\Sigma_\infty}(V)) = R(L_{\Sigma_0}^{\Sigma_\infty}(V))$.
			\item If $\Sigma_\infty\Subset V$ then $R(L_{\Sigma_0}^{\emptyset}(V)) = R(L_{\Sigma_0}^{\Sigma_\infty}(V))$.
		\end{enumerate}
	\end{proposition}
	\begin{proof}
		We give a proof of (i), and note that the proof of (ii) is almost the same with obvious modifications.
		
		Proof of ``$\subseteq$'': Let $f\in R(L_{\emptyset}^{\Sigma_\infty}(V))$, then there exists $F_1\in L^2(V)^d$ such that $f = w|_\Gamma$ with $w = w_{\emptyset,F_1}^{\Sigma_\infty}$. We have
		\begin{equation} \label{eq:weq1}
			\int_{\Omega} \gamma_0\nabla w\cdot\nabla v' \,\di x = \int_{V} F_1\cdot\nabla v'\,\di x, \qquad \forall v'\in \Hs_{\emptyset}^{\Sigma_\infty}.
		\end{equation}
		
		Let $W_1, W_2 \in \A$ such that $\Sigma_0 \Subset W_1 \Subset W_2 \subseteq V$ and with $\dist(W_2,\Sigma_\infty)>0$. Using the $C^\infty$-Urysohn lemma \cite[Lemma~8.18]{folland}, we may construct $\varphi \in C_{\textup{c}}^\infty(W_2)$ with $0\leq\varphi\leq 1$ and $\varphi \equiv 1$ in $W_1$. In the following we will let $v\in\Hs_{\Sigma_0}^{\Sigma_\infty}$ be arbitrary but fixed, and let $c$ be the constant such that
		\begin{equation*}
			v_0 = v-c
		\end{equation*}
		has vanishing mean on $\partial V$. Moreover, we also define
		\begin{equation*}
			v' = (1-\varphi)v_0.
		\end{equation*}
		We note that $v'$ (extended by zero to ${\overline{\Sigma_0}}$)	lies in $\Hs^{\Sigma_\infty}_{\emptyset}$. Therefore $\eqref{eq:weq1}$ is applicable for $v'$, and using the fact that $v'|_{\Omega\setminus V} = v_0|_{\Omega\setminus V}$ we now get
		\begin{align}
			\int_{\Omega} \gamma_0\nabla w\cdot\nabla v\,\di x &= \int_{V} \gamma_0\nabla {w}\cdot\nabla v_0\,\di x + \int_{\Omega\setminus V} \gamma_0\nabla w\cdot\nabla v'\,\di x \notag \\
			&= \int_{V} \gamma_0\nabla w\cdot\nabla v_0\,\di x + \int_{\Omega} \gamma_0\nabla w\cdot\nabla v'\,\di x - \int_{V} \gamma_0\nabla w\cdot\nabla v'\,\di x \notag \\
			&= \int_{V} \gamma_0\nabla w\cdot\nabla v_0\,\di x + \int_V (F_1 - \gamma_0\nabla w)\cdot\nabla v'\,\di x. \label{eq:tmpintw1}
		\end{align}
		Define $F_2 = F_1 - \gamma_0\nabla w|_V$. Using the definitions of ${v'}$ and $\varphi$, and that $\Sigma_0\Subset W_1$ where $1-\varphi$ vanishes, we have
		\begin{equation} \label{eq:tmpintw2}
			\int_V F_2\cdot{\nabla {v'}}\,\di x = \int_{V} (1-\varphi)F_2\cdot\nabla v_0\,\di x + \int_{V} (F_2\cdot\nabla(1-\varphi))v_0\,\di x. 
		\end{equation}
		Define $F_3 = \gamma_0\nabla {w}|_V + (1-\varphi)F_2$ and $g = F_2\cdot\nabla(1-\varphi)$, then \eqref{eq:tmpintw1} and \eqref{eq:tmpintw2} collectively give
		\begin{equation} \label{eq:tmpintw3}
			\int_{\Omega} \gamma_0\nabla {w}\cdot{\nabla v}\,\di x = \int_V F_3\cdot{\nabla v_0}\,\di x + \int_{V} gv_0\,\di x.
		\end{equation}
		We now consider an auxiliary problem. Let $V^\circ$ be the interior of $V$ and let
		\begin{equation*}
			\Hs_V = \Bigl\{\widehat{v}\in H^1(V^\circ\setminus\overline{\Sigma_0}) \mid \int_{\partial V} \widehat{v}|_{\partial V}\,\di S = 0 \Bigr\},
		\end{equation*}
		then $\Hs_V$ is a Hilbert space with inner product (whose norm is equivalent to the usual $H^1$-norm)
		\begin{equation*}
			\inner{\widehat{v}_1,\widehat{v}_2}_{\Hs_V} = \int_V \nabla \widehat{v}_1\cdot\nabla\widehat{v}_2\,\di x.
		\end{equation*}
		Using the Lax--Milgram lemma, we now define  $\widehat{w}$ to be the unique solution in $\Hs_V$ of the variational problem
		\begin{equation} \label{eq:tmpintw4}
			\int_{V} \nabla \widehat{w}\cdot {\nabla\widehat{v}}\,\di x = \int_{V} g \widehat{v}\,\di x, \qquad \forall \widehat{v}\in\Hs_V.
		\end{equation}
		Since $v_0$ has vanishing mean on $\partial V$ the restriction of $v_0$ to $V$ belongs to $\Hs_V$. Let $F = F_3 + \nabla\widehat{w}$, then \eqref{eq:tmpintw3} and \eqref{eq:tmpintw4} lead to
		\begin{equation*}
			\int_{\Omega} \gamma_0\nabla w\cdot\nabla v\,\di x = \int_V F\cdot\nabla v_0\,\di x = \int_V F\cdot\nabla v\,\di x.
		\end{equation*}
		As $v\in\Hs_{\Sigma_0}^{\Sigma_\infty}$ was arbitrary (and $w \in \Hs_{\emptyset}^{\Sigma_\infty} \subseteq \Hs_{\Sigma_0}^{\Sigma_\infty}$), we have proven 
		\begin{equation*}
			f = w|_\Gamma  = L_{\Sigma_0}^{\Sigma_\infty}(V)F,
		\end{equation*}
		and therefore $R(L_{\emptyset}^{\Sigma_\infty}(V)) \subseteq R(L_{\Sigma_0}^{\Sigma_\infty}(V))$.
		
		\bigskip
		
		Proof of ``$\supseteq$'': Suppose now that $f\in R(L_{\Sigma_0}^{\Sigma_\infty}(V))$, hence there exists $F_1\in L^2(V)^d$ such that $f = w|_\Gamma$ with $w = w_{\Sigma_0,F_1}^{\Sigma_\infty}$. Thus we have
		\begin{equation} \label{eq:weq2}
			\int_{\Omega} \gamma_0\nabla w\cdot\nabla v \,\di x = \int_{V} F_1\cdot\nabla v\,\di x, \qquad \forall v\in \Hs_{\Sigma_0}^{\Sigma_\infty}.
		\end{equation} 
		
		Let $\varphi$ be defined as before, and define $\widetilde{w}\in \Hs_{\emptyset}^{\Sigma_\infty}$ via the formula $\widetilde{w}=(1-\varphi) w$ and extended by zero on $\overline{\Sigma_0}$. Recall that $\Hs_{\emptyset}^{\Sigma_\infty}\subseteq\Hs_{\Sigma_0}^{\Sigma_\infty}$ and let $v\in \Hs_{\emptyset}^{\Sigma_\infty}$ be arbitrary, then \eqref{eq:weq2} implies 
		\begin{align*}
			\int_{\Omega} \gamma_0\nabla \widetilde{w}\cdot\nabla v \,\di x &= \int_{\Omega} \gamma_0\nabla w\cdot\nabla v \,\di x - \int_{V} \gamma_0\nabla (\varphi w)\cdot\nabla v \,\di x \\
			&= \int_V (F_1 - \gamma_0\nabla (\varphi w))\cdot\nabla v\,\di x.
		\end{align*}
		Hence, with $F = F_1 - \gamma_0\nabla (\varphi w)|_{V}$  
		\begin{equation*}
			f = w|_\Gamma = \widetilde{w}|_\Gamma = L_{\emptyset}^{\Sigma_\infty}(V)F,
		\end{equation*}
		and we have verified that $R(L_{\emptyset}^{\Sigma_\infty}(V)) \supseteq R(L_{\Sigma_0}^{\Sigma_\infty}(V))$.
	\end{proof}
	
	\begin{remark} \label{remark:rangeindep}
		If $\Sigma\Subset V$ in Proposition~\ref{prop:rangeindep}, then one can combine (i) and (ii) to obtain:
		\begin{equation*}
			R(L_{\emptyset}^{\emptyset}(V)) = R(L_{\Sigma_0}^{\emptyset}(V)) = R(L_{\emptyset}^{\Sigma_\infty}(V)) =  R(L_{\Sigma_0}^{\Sigma_\infty}(V)).
		\end{equation*}
		Indeed, after using (i) and (ii) as they are stated, we can use (i) again for a $(\Sigma_0',\Sigma_\infty')$ collection of cracks with $\Sigma_0' = \Sigma_0$ and $\Sigma_\infty' = \emptyset$.
	\end{remark}
	
	The following result will be needed in the proof of Proposition \ref{prop:locpot}. For reasons of clarity we state it as a separate lemma with its own proof. The proof makes essential use of single and double layer potentials; for an in-depth analysis of these in the context of Lipschitz surfaces we refer to \cite{Verchota1984,McLean}.\footnote{The additional regularity required of $\gamma_0$ in this lemma guarantees the validity of the ``appropriately rescaled" jump relations for the (single and double) layer potentials, typically stated in terms of the fundamental solution for the Laplacian.}
	
	\begin{lemma}\label{lemma:nontrivial}
	Let $\Sigma$ be a $(\Sigma_0,\Sigma_\infty)$ collection of cracks. Assume that $\gamma_0 \in C^2(\overline{\Omega})$ and is positive. 
	\begin{enumerate}[\rm(i)]
		\item If $\Sigma_0\neq \emptyset$ then there exists $f\in L^2_\diamond(\Gamma)$ such that $u_{\emptyset,f}^{\Sigma_\infty} \neq u_{\Sigma_0,f}^{\Sigma_\infty}$.
		\item If $\Sigma_\infty\neq \emptyset$ then there exists $f\in L^2_\diamond(\Gamma)$ such that $u_{\Sigma_0,f}^{\emptyset} \neq u_{\Sigma_0,f}^{\Sigma_\infty}$.
	\end{enumerate}
	\end{lemma}
	\begin{proof}
	Proof of (i): It suffices to prove that there exists an $f\in L^2_\diamond(\Gamma)$ such that 
	\begin{equation*}
		\frac{\partial}{\partial n}u_{\emptyset,f}^{\Sigma_\infty} \text{ does not vanish identically on } \Sigma_0.
	\end{equation*}
	Let $\Phi(x,y)$ denote the (fundamental) solution to the problem
	\begin{align*}
		-\nabla\cdot(\gamma_0\nabla \Phi(\,\cdot\,,y)) &= \delta_y \quad \text{in } \Omega\setminus \overline{\Sigma_\infty}, \\
		\gamma_0\frac{\partial}{\partial \nu}\Phi(\,\cdot\,,y) &= -\frac{1}{|\partial \Omega|} \quad \text{on } \partial \Omega, \\
		  & \hspace{-2.2cm} \Phi(\,\cdot\,,y) \text{ is locally constant on } \Sigma_\infty, \\
		\int_{\Sigma_i} \Bigl[\gamma_0\frac{\partial}{\partial n}\Phi(\,\cdot\,,y)\Bigr]\di S &= 0 \quad \text{for each component $\Sigma_i$ of $\Sigma_\infty$}.
     \end{align*}
	In terms of this $\Phi$ we have the representation formula
	\begin{equation*}
		u_{\emptyset,f}^{\Sigma_\infty}(y) =\int_{\Gamma} \Phi(x,y)f(x)\,\di S_x,
	\end{equation*}
	and so
	\begin{equation*}
		\frac{\partial}{\partial n} u_{\emptyset,f}^{\Sigma_\infty}(y) = \int_{\Gamma} \frac{\partial}{\partial n_y}\Phi(x,y)f(x)\,\di S_x \quad \text{on } \Sigma_0.
	\end{equation*}
	Now suppose that $\frac{\partial}{\partial n} u_{\emptyset,f}^{\Sigma_\infty}$ vanishes identically on $\Sigma_0$ for all $f \in  L^2_\diamond(\Gamma)$. Let $\psi$ be a non-zero continuous function on $\Sigma_0$, supported away from $\partial \Sigma_0$, then it would follow that
	\begin{equation*}
	0 = \int_{\Sigma_0} \psi(y)\int_{\Gamma} \frac{\partial}{\partial n_y}\Phi(x,y)f(x)\,\di S_x\,\di S_y = \int_{\Gamma} f(x) \int_{\Sigma_0} \frac{\partial}{\partial n_y}\Phi(x,y) \psi(y)\,\di S_y\,\di S_x
	\end{equation*}
	for all $f\in L^2_\diamond(\Gamma)$. In other words, there is a constant $C$ such that
	\begin{equation*}
		\int_{\Sigma_0} \frac{\partial}{\partial n_y}\Phi(x,y) \psi(y)\,\di S_y \equiv C \quad \text{on } \Gamma.
	\end{equation*}
	We also have that 
	\begin{equation*}
		\gamma_0(x)\frac{\partial}{\partial \nu}\int_{\Sigma_0} \frac{\partial}{\partial n_y}\Phi(x,y) \psi(y)\,\di S_y= \int_{\Sigma_0} \frac{\partial}{\partial n_y}\Bigl(\gamma_0(x)\frac{\partial}{\partial \nu_x} \Phi(x,y)\Bigr) \psi(y)\,\di S_y \equiv 0 \quad \text{on } \Gamma.
	\end{equation*}
	By unique continuation of the operator $\nabla\cdot (\gamma_0\nabla (\,\cdot\,))$ it follows that 
	\begin{equation*}
		\int_{\Sigma_0} \frac{\partial}{\partial n_y}\Phi(x,y) \psi(y)\,\di S_y \equiv C \quad \text{in }\Omega \setminus \overline{\Sigma_0}.
	\end{equation*}
	The jump relation for double layer potentials (see \cite{Verchota1984,McLean}) now asserts that
	$$
	\gamma_0(x)^{-1}\psi(x) = \left[ \int_{\Sigma_0} \frac{\partial}{\partial n_y}\Phi(x,y) \psi(y)\,\di S_y  \right] = 0 ~\hbox{ on } \Sigma_0.
	$$
	This is contrary to the initial assumption that $\psi$ is non-zero.\footnote{The $C^2$ regularity of $\gamma_0$ is sufficient to guarantee that the fundamental solution $\Phi(x,y)$ has the form $\Phi(x,y)= \gamma_0(y)^{-1}\Phi_{\Delta}(x,y) +R(x,y)=\gamma_0(x)^{-1}\Phi_{\Delta}(x,y) +\widetilde R(x,y)$, where $\Phi_\Delta$ is the standard Green's function for the Laplacian, and the ``regular terms" $R$ and $\widetilde R $ do not contribute to the jump relations.}
	
	\bigskip
	
	Proof of (ii): It suffices to prove that there exists an $f\in L^2_\diamond(\Gamma)$ such that 
	\begin{equation*}
		u_{\Sigma_0,f}^{\emptyset} \text{ is not locally constant on } \Sigma_\infty.
	\end{equation*}
	Let $\Psi(x,y)$ denote the (fundamental) solution to the problem
	\begin{align*}
		-\nabla\cdot(\gamma_0\nabla \Psi(\,\cdot\,,y)) &= \delta_y \quad \text{in } \Omega\setminus \overline{\Sigma_0}, \\
		\gamma_0\frac{\partial}{\partial \nu}\Psi(\,\cdot\,,y) &= -\frac{1}{|\partial \Omega|} \quad \text{on } \partial \Omega, \\
		\gamma_0\frac{\partial}{\partial n}\Psi(\,\cdot\,,y) &= 0 \quad \text{on } \Sigma_0. 
	\end{align*}
	We have the representation formula
	\begin{equation*}
		u_{\Sigma_0,f}^{\emptyset}(y) =\int_{\Gamma} \Psi(x,y)f(x)\,\di S_x.
	\end{equation*}
	Now suppose that $u_{\Sigma_0,f}^{\emptyset}$ is locally constant on $\Sigma_\infty$ for all $f \in  L^2_\diamond(\Gamma)$. Let $\varphi$ be a non-zero continuous function on $\Sigma_\infty$, supported away from $\partial\Sigma_\infty$, and with vanishing mean on each component of $\Sigma_\infty$ (i.e.\ orthogonal to $u_{\Sigma_0,f}^{\emptyset}$ on $\Sigma_\infty$ in the $L^2$ inner product), then it would follow that
	\begin{equation*}
		0 = \int_{\Sigma_\infty} \varphi(y)\int_{\Gamma} \Psi(x,y)f(x)\,\di S_x\,\di S_y = \int_{\Gamma} f(x) \int_{\Sigma_\infty} \Psi(x,y) \varphi(y)\,\di S_y\,\di S_x
	\end{equation*}
	for all $f\in L^2_\diamond(\Gamma)$. Thus, there is a constant $C$ such that
	\begin{equation*}
		\int_{\Sigma_\infty} \Psi(x,y) \varphi(y)\,\di S_y \equiv C \quad \text{on } \Gamma.
	\end{equation*}
	Since $\varphi$ has vanishing mean on $\Sigma_\infty$, we also have that 
	\begin{equation*}
		\gamma_0(x)\frac{\partial}{\partial \nu}\int_{\Sigma_\infty} \Psi(x,y) \varphi(y)\,\di S_y = \int_{\Sigma_\infty} \Bigl(\gamma_0(x)\frac{\partial}{\partial \nu_x} \Psi(x,y)\Bigr) \varphi(y)\,\di S_y \equiv 0 \quad \text{on } \Gamma.
	\end{equation*}
	By unique continuation, it follows that 
	\begin{equation*}
		\int_{\Sigma_\infty} \Psi(x,y) \varphi(y)\,\di S_y \equiv C \quad \text{in }\Omega \setminus \overline{\Sigma}.
	\end{equation*}
	The jump relation for the normal derivative of single layer potentials (see \cite{Verchota1984,McLean}) now asserts that
	$$
	-\gamma_0(x)^{-1}\varphi(x)= \left[ \frac{\partial }{\partial n}\int_{\Sigma_\infty} \Psi(x,y) \varphi(y)\,\di S_y \right]=0~\hbox{ on } \Sigma_\infty .
	$$
	The fact that $\varphi$ must identically vanish on $\Sigma_\infty$ represents a contradiction. 
	\end{proof}

	\begin{proposition} \label{prop:locpot}
		Let $\Sigma$ be a $(\Sigma_0,\Sigma_\infty)$ collection of cracks. Assume that $\Sigma_0\Subset V$ and $\Sigma_\infty \Subset W$ for $V,W \in\A$ with $\dist(V,W)>0$. 
		\begin{enumerate}[\rm(i)]
			\item If $\Sigma_0\neq\emptyset$ then there exists a sequence $(f_n)$ in $L^2_\diamond(\Gamma)$ such that
			\begin{align*}
				\lim_{n\to\infty}\inner{(\Lambda_{W}^{\emptyset} - \Lambda_{\emptyset}^{\emptyset})f_n,f_n} &= 0, \\
				\lim_{n\to\infty}\inner{(\Lambda_{\emptyset}^{\emptyset} - \Lambda_{\emptyset}^{W})f_n,f_n} &= 0, \\
				\lim_{n\to\infty}\inner{(\Lambda_{\Sigma_0}^{\Sigma_\infty}-\Lambda_{\emptyset}^{\Sigma_\infty})f_n,f_n} &= \infty.
			\end{align*}
			\item If $\Sigma_\infty\neq\emptyset$ then there exists a sequence $(g_n)$ in $L^2_\diamond(\Gamma)$ such that
			\begin{align*}
				\lim_{n\to\infty}\inner{(\Lambda_{V}^{\emptyset} - \Lambda_{\emptyset}^{\emptyset})g_n,g_n} &= 0, \\
				\lim_{n\to\infty}\inner{(\Lambda_{\emptyset}^{\emptyset} - \Lambda_{\emptyset}^{V})g_n,g_n} &= 0, \\
				\lim_{n\to\infty}\inner{(\Lambda_{\Sigma_0}^{\emptyset}-\Lambda_{\Sigma_0}^{\Sigma_\infty})g_n,g_n} &= \infty.
			\end{align*}
		\end{enumerate}
	\end{proposition}

	\begin{proof}
		We give a proof of (i), and note that the proof of (ii) is almost the same with obvious modifications. 
		
		The first part of this proof, up to \eqref{eq:rangecap}, is a slight modification of \cite[Proof of Theorem~3.6]{Harrach13}, and the arguments are essentially the same, however, our setup includes perfectly conducting cracks. Due to the notational differences we give the full details. 
		
		Let $Y\in\mathcal{A}$ such that $W\Subset Y$ and $\dist(V,Y)>0$. We will construct a localisation for certain electric potentials, such that their energy tends to zero on the slightly larger set $Y$ instead of just on $W$. We make use of this fact to establish \eqref{eq:Wloc} towards the end of the proof. 
		
		Let $U = \overline{\Omega}\setminus (V\cup Y)$ then $U$ is connected, $U$ is relatively open in $\overline{\Omega}$, and $\partial\Omega\subset U$. For any 
		\begin{equation*}
			h\in R(L_\emptyset^{\Sigma_\infty}(V)) \cap R(L_\emptyset^{\Sigma_\infty}(Y))
		\end{equation*}
		there exist $F_V\in L^2(V)^d$ and $F_Y\in L^2(Y)^d$ such that $h = L_\emptyset^{\Sigma_\infty}(V)F_V = L_\emptyset^{\Sigma_\infty}(Y)F_Y$. Let $u_V = w_{\emptyset,F_V}^{\Sigma_\infty}$ and $u_Y = w_{\emptyset,F_Y}^{\Sigma_\infty}$. Thus $u_V,u_Y\in \Hs_{\emptyset}^{\Sigma_\infty}$ and from the variational formulations we obtain the equations
		\begin{align*}
			-\nabla \cdot(\gamma_0 \nabla u_V) &= 0 \quad \text{in } \Omega\setminus (V\cup \overline{\Sigma_\infty}), \\
			-\nabla \cdot(\gamma_0 \nabla u_Y) &= 0 \quad \text{in } \Omega\setminus Y, \\
			u_V = u_Y &= h \quad \text{on } \Gamma, \\
			\gamma_0\frac{\partial u_V}{\partial\nu} = \gamma_0\frac{\partial u_Y}{\partial\nu} &= 0 \quad \text{on } \partial\Omega.
		\end{align*}
		The variational formulation also gives that
		\begin{equation*}
			\int_{\Sigma_i}\Bigl[ \gamma_0 \frac{\partial u_V}{\partial n}\Bigr] \di S =0 \quad \text{for each component $\Sigma_i$ of $\Sigma_\infty$.}
		\end{equation*}
		Due to the  unique continuation property (UCP) we get that $u_V|_U = u_Y|_U$. By gluing $u_V$ and $u_Y$ together, we obtain the function
		\begin{equation*}
			u = \begin{cases}
				u_V & \text{in } \Omega\setminus V, \\
				u_Y & \text{in } V,
			\end{cases}
		\end{equation*} 
		which, due to the boundary  value problem it solves, must equal $u_{\emptyset,0}^{\Sigma_\infty}=0$. Consequently $h = u|_\Gamma = 0$, and so we have proven
		\begin{equation} \label{eq:rangecap}
			R(L_\emptyset^{\Sigma_\infty}(V)) \cap R(L_\emptyset^{\Sigma_\infty}(Y)) = \{0\}.
		\end{equation}
		%
		%
		%
		Since $\Sigma_0\Subset V$, we get from Proposition~\ref{prop:rangeindep}(i) that
		\begin{equation} \label{eq:rangeeq}
			R(L_{\emptyset}^{\Sigma_\infty}(V)) = R(L_{\Sigma_0}^{\Sigma_\infty}(V)).
		\end{equation}
		Define 
		\begin{equation*}
			A = L_{\Sigma_0}^{\Sigma_\infty}(V) - L_\emptyset^{\Sigma_\infty}(V).
		\end{equation*}
		Due to \eqref{eq:Ladj} we have 
		\begin{equation} \label{eq:Astar}
			A^*f = (L_{\Sigma_0}^{\Sigma_\infty}(V))^*f - (L_{\emptyset}^{\Sigma_\infty}(V))^*f = \nabla (u^{\Sigma_\infty}_{\Sigma_0, f} - u^{\Sigma_\infty}_{\emptyset, f})|_V.
		\end{equation}
		Using unique continuation and the mean-free conditions on $\Gamma$ for $u^{\Sigma_\infty}_{\Sigma_0, f}$ and $u^{\Sigma_\infty}_{\emptyset, f}$, we now conclude that
		\begin{equation*}
			A^*f = 0 \qquad \text{if and only if} \qquad u^{\Sigma_\infty}_{\Sigma_0, f} = u^{\Sigma_\infty}_{\emptyset, f}.
		\end{equation*}
		Since $\Sigma_0\neq\emptyset$, lemma~\ref{lemma:nontrivial}(i) implies $A^* \neq 0$ and therefore $A \neq 0$. As a consequence there exists $g\in R(A)\setminus\{0\}$, and because of \eqref{eq:rangecap} and \eqref{eq:rangeeq} this $g$ satisfies $g\in R(L_\emptyset^{\Sigma_\infty}(V))$ and $g\not\in R(L_\emptyset^{\Sigma_\infty}(Y))$. Moreover, $(L_\emptyset^{\Sigma_\infty}(Y))^*$ is injective by \eqref{eq:Ladj} and unique continuation. We then apply the constructive result on localised potentials in Lemma~\ref{lemma:locgen}, and for $n\in\N$ define
		\begin{equation*}
			\xi_n = \bigl(L_\emptyset^{\Sigma_\infty}(Y)(L_\emptyset^{\Sigma_\infty}(Y))^* + \tfrac{1}{n}I\bigr)^{-1}g
		\end{equation*}
		and
		\begin{equation*}
			f_n =  \frac{\xi_n}{\norm{(L_\emptyset^{\Sigma_\infty}(Y))^*\xi_n}_{L^2(Y)^d}^{3/2}}.
		\end{equation*}
		Lemma~\ref{lemma:locgen} implies that 
		\begin{align}
			\lim_{n\to\infty}\norm{(L_\emptyset^{\Sigma_\infty}(Y))^*f_n}_{L^2(Y)^d} &= 0, \label{eq:wlim1}\\
			\lim_{n\to\infty}\norm{(L_\emptyset^{\Sigma_\infty}(V))^*f_n}_{L^2(V)^d} &= \infty, \label{eq:vlim1}\\
			\text{and } ~\lim_{n\to\infty}\norm{A^*f_n}_{L^2(V)^d} &= \infty, \label{eq:Jlim}
		\end{align}
		simultaneously.	Since $\Sigma_\infty\Subset Y$, Proposition~\ref{prop:rangeindep}(ii) gives that $R(L_\emptyset^{\Sigma_\infty}(Y)) = R(L_\emptyset^{\emptyset}(Y))$. This range equality, together with \eqref{eq:wlim1} and Lemma~\ref{lemma:rangenorm}, implies
		\begin{equation} \label{eq:wlim2}
			\lim_{n\to\infty}\norm{(L_\emptyset^{\emptyset}(Y))^*f_n}_{L^2(Y)^d} = 0.
		\end{equation}
		From Proposition~\ref{prop:rangeindep}(ii) we have $R(L_\emptyset^{\Sigma_\infty}(V\cup Y)) = R(L_\emptyset^{\emptyset}(V\cup Y))$. So from \eqref{eq:Ladj} and Lemma~\ref{lemma:rangenorm} there exists a $c>0$ such that
		\begin{align}
			\norm{(L_\emptyset^{\emptyset}(V))^*f_n}_{L^2(V)^d}^2 + \norm{(L_\emptyset^{\emptyset}(Y))^*f_n}_{L^2(Y)^d}^2 \hspace{-3cm}&\notag\\
			&= \norm{(L_\emptyset^{\emptyset}(V\cup Y))^*f_n}_{L^2(V\cup Y)^d}^2 \notag\\
			&\geq c\norm{(L_\emptyset^{\Sigma_\infty}(V\cup Y))^*f_n}_{L^2(V\cup Y)^d}^2 \notag\\
			&\geq c\norm{(L_\emptyset^{\Sigma_\infty}(V))^*f_n}_{L^2(V)^d}^2. \label{eq:VWineq}
		\end{align}
		A combination of  \eqref{eq:vlim1}, \eqref{eq:wlim2}, and \eqref{eq:VWineq} gives
		\begin{equation} \label{eq:vlim2}
			\lim_{n\to\infty}\norm{(L_\emptyset^{\emptyset}(V))^*f_n}_{L^2(V)^d} = \infty.
		\end{equation}
		
		We are now ready to associate the different operators to differences of ND mappings. Let $\widehat{u}_n = u_{\emptyset,f_n}^{\emptyset}$. Then \eqref{eq:Ladj}, \eqref{eq:wlim2}, and \eqref{eq:vlim2} translate into 
		\begin{equation} \label{eq:locbg}
			\lim_{n\to\infty} \int_{V} \abs{\nabla \widehat{u}_n}^2\,\di x = \infty \qquad \text{and} \qquad \lim_{n\to\infty} \int_{Y} \abs{\nabla \widehat{u}_n}^2\,\di x = 0.
		\end{equation}
		Using \cite[Lemma~A.1(ii)]{Garde2020}, we obtain the existence of a constant $K>0$, independent of $f_n$, such that
		\begin{equation} \label{eq:Weq1}
			0 \leq \inner{(\Lambda_{\emptyset}^{\emptyset} - \Lambda_{\emptyset}^{W})f_n,f_n} \leq K\int_W \abs{\nabla \widehat{u}_n}^2\,\di x.
		\end{equation}
		Let $\breve{u}_n$ equal $u_{W,f_n}^{\emptyset}$ in $\Omega\setminus W$ and satisfy the following Dirichlet problem in $W$:
		\begin{align*}
			-\nabla \cdot(\gamma_0\nabla \breve{u}_n) &= 0 \quad \text{in the interior of } W, \\
			\breve{u}_n &= u_{W,f_n}^{\emptyset} \quad \text{on } \partial W. 
		\end{align*}
		The result in \cite[Lemma~5.3]{Garde2020} ensures that if an extreme inclusion (either perfectly insulating or perfectly conducting) is introduced, \emph{compactly contained} in $Y$ where the energy of $\widehat{u}_n$ tends to zero, then the corresponding electric potentials with this new conductivity profile will have the same localisation as in \eqref{eq:locbg}. For a perfectly insulating inclusion, $W \Subset Y$, the localisation applies to $\breve{u}_n$, using precisely the extension in $W$ defined above. We thus have
		\begin{equation} \label{eq:Wloc}
			\lim_{n\to\infty} \int_{Y} \abs{\nabla \breve{u}_n}^2\,\di x = 0.
		\end{equation}
		Use of \cite[Lemma~A.1(iii)]{Garde2020} yields
		\begin{equation} \label{eq:Weq2}
			0 \leq \inner{(\Lambda_{W}^{\emptyset} - \Lambda_{\emptyset}^{\emptyset})f_n,f_n} \leq \esssup(\gamma_0)\int_W \abs{\nabla \breve{u}_n}^2\,\di x.
		\end{equation}
		From \eqref{eq:locbg}--\eqref{eq:Weq2} we now conclude
		\begin{equation*}
			\lim_{n\to\infty}\inner{(\Lambda_{W}^{\emptyset} - \Lambda_{\emptyset}^{\emptyset})f_n,f_n} = \lim_{n\to\infty}\inner{(\Lambda_{\emptyset}^{\emptyset} - \Lambda_{\emptyset}^{W})f_n,f_n} = 0.
		\end{equation*}
	
		It remains to prove that $\inner{(\Lambda_{\Sigma_0}^{\Sigma_\infty}-\Lambda_{\emptyset}^{\Sigma_\infty})f_n,f_n}\to\infty$ for $n\to\infty$. Let $u_n = u_{\Sigma_0,f_n}^{\Sigma_\infty}$ and $\widetilde{u}_n = u_{\emptyset,f_n}^{\Sigma_\infty}$, then from Lemma~\ref{lemma:NDdiff}(i) we have
		\begin{equation*}
			\inner{(\Lambda_{\Sigma_0}^{\Sigma_\infty}-\Lambda_{\emptyset}^{\Sigma_\infty})f_n,f_n} = \norm{u_n-\widetilde{u}_n}_{*}^2 \geq \essinf(\gamma_0)\int_{V} \abs{\nabla(u_n-\widetilde{u}_n)}^2\,\di x.
		\end{equation*}
		At the same time from \eqref{eq:Astar} 
		\begin{equation*}
			A^*f_n = \nabla (u_n-\widetilde{u}_n)|_V,
		\end{equation*}
		so use of \eqref{eq:Jlim} concludes the proof.
	\end{proof}
	
	\section{The direction ``$\Leftarrow$'' in Theorems~\ref{thm:main} \& \ref{thm:main2}} \label{sec:proof_main2}
	
	Finally we show the ``difficult direction'' of the \emph{if and only if} statements of Theorems~\ref{thm:main} and \ref{thm:main2}, although the majority of the work has already been done in Proposition~\ref{prop:locpot}.
	
	\begin{proposition} \label{prop:main2}
		Let $D$ be a $(D_0,D_\infty)$ collection of cracks. Given any $C\in\A$, then
		\begin{equation*}
			\Lambda_{C}^{\emptyset}\geq \Lambda_{D_0}^{D_\infty} \geq \Lambda_{\emptyset}^C \qquad \text{implies} \qquad D\subset C.
		\end{equation*}
	\end{proposition}
	\begin{proof}
		We will prove the contrapositive statement, i.e.~assume $D\not\subset C$. As $C$ is closed there is a relatively open part of $D$ that belongs to $\Omega\setminus C$. Now we will show that either $\Lambda_{C}^{\emptyset}\not\geq\Lambda_{D_0}^{D_\infty}$ or $\Lambda_{D_0}^{D_\infty}\not\geq \Lambda_{\emptyset}^C$. 
		
		Since $C$ has connected complement we have either of two cases, which will dictate which operator inequality that will fail to hold:
		\begin{enumerate}[(a)]
			\item There are $V,W\in \A$ with $\dist(V,W)>0$ and non-empty $\chi\in\X$, such that
			\begin{equation*}
				\chi\subseteq D_0, \qquad \chi\Subset V, \qquad C\subseteq W, \qquad \text{and} \qquad D_\infty\Subset W.
			\end{equation*}
			\item There are $V,W\in \A$ with $\dist(V,W)>0$ and non-empty $\chi\in\X$, such that 
			\begin{equation*}
				\chi\subseteq D_\infty, \qquad \chi \Subset W, \qquad C\subseteq V, \qquad \text{and} \qquad D_0\Subset V.
			\end{equation*}
		\end{enumerate}
		
		
		Case (a): Since $C\cup D_\infty\subseteq W$ then \cite[Theorem~3.7]{Garde2020} and (the second inequality of) Proposition~\ref{prop:main1} give
		\begin{equation*}
			\Lambda_{C}^{\emptyset} \leq \Lambda_{W}^{\emptyset} \qquad \text{and} \qquad \Lambda_\emptyset^{D_\infty} \geq \Lambda_\emptyset^{W}.
		\end{equation*}
		As $\chi\subseteq D_0$ then Proposition~\ref{prop:main12}(i) gives
		\begin{equation*}
			\Lambda_{D_0}^{D_\infty} \geq \Lambda_{\chi}^{D_\infty}.
		\end{equation*}
		Hence we have
		\begin{align*}
			\Lambda_{C}^{\emptyset} - \Lambda_{D_0}^{D_\infty} &= (\Lambda_{C}^{\emptyset} - \Lambda_\emptyset^\emptyset) + (\Lambda_{\emptyset}^\emptyset - \Lambda_\emptyset^{D_\infty}) + (\Lambda_\emptyset^{D_\infty} - \Lambda_{D_0}^{D_\infty}) \\
			&\leq (\Lambda_{W}^{\emptyset} - \Lambda_\emptyset^\emptyset) + (\Lambda_{\emptyset}^\emptyset - \Lambda_\emptyset^{W}) + (\Lambda_\emptyset^{D_\infty} - \Lambda_{\chi}^{D_\infty}).
		\end{align*}
		Now, since $\emptyset\neq\chi\Subset V$ and $D_\infty\Subset W$, Proposition~\ref{prop:locpot}(i) gives the existence of a sequence $(f_n)$ in $L^2_\diamond(\Gamma)$ such that
		\begin{equation*}
			\lim_{n\to\infty} \inner{(\Lambda_{C}^{\emptyset} - \Lambda_{D_0}^{D_\infty})f_n,f_n} = -\infty,
		\end{equation*}
		and in particular we conclude that $\Lambda_{C}^{\emptyset}\not\geq \Lambda_{D_0}^{D_\infty}$. 
		
		Case (b): Since $C\cup D_0\subseteq V$ then \cite[Theorem~3.7]{Garde2020} and (the first inequality of) Proposition~\ref{prop:main1} gives
		\begin{equation*}
			\Lambda_\emptyset^C \geq \Lambda_{\emptyset}^V \qquad \text{and} \qquad \Lambda_{D_0}^{\emptyset} \leq \Lambda_V^{\emptyset}.
		\end{equation*}
		As $\chi\subseteq D_\infty$ then Proposition~\ref{prop:main12}(ii) gives
		\begin{equation*}
			\Lambda_{D_0}^{D_\infty} \leq \Lambda_{D_0}^{\chi}.
		\end{equation*}
		Hence we have
		\begin{align*}
			\Lambda_{D_0}^{D_\infty} - \Lambda_\emptyset^C &= (\Lambda_{D_0}^{D_\infty} - \Lambda_{D_0}^{\emptyset}) + (\Lambda_{D_0}^{\emptyset} - \Lambda_\emptyset^\emptyset) + (\Lambda_\emptyset^\emptyset - \Lambda_\emptyset^C) \\
			&\leq (\Lambda_{D_0}^{\chi} - \Lambda_{D_0}^{\emptyset}) + (\Lambda_{V}^{\emptyset} - \Lambda_\emptyset^\emptyset) + (\Lambda_\emptyset^\emptyset - \Lambda_\emptyset^V).
		\end{align*}
		Now, since $\emptyset\neq\chi\Subset W$ and $D_0\Subset V$, Proposition~\ref{prop:locpot}(ii) gives the existence of a sequence $(g_n)$ in $L^2_\diamond(\Gamma)$ such that
		\begin{equation*}
			\lim_{n\to\infty} \inner{(\Lambda_{D_0}^{D_\infty} - \Lambda_\emptyset^C)g_n,g_n} = -\infty,
		\end{equation*}
		and in particular we conclude that $\Lambda_{D_0}^{D_\infty}\not\geq \Lambda_\emptyset^C$.
	\end{proof}
	
	\begin{proposition} \label{prop:main22}
		Let $D\in \X$.
		\begin{enumerate}[\rm(i)]
			\item Given any $\chi\in\X$, then
			\begin{equation*}
				\Lambda_{D}^{\emptyset} \geq \Lambda_{\chi}^{\emptyset} \qquad \text{implies} \qquad \chi\subseteq D.
			\end{equation*}
			\item Given any $\chi\in\X$, then
			\begin{equation*}
				\Lambda_{\emptyset}^{\chi} \geq \Lambda_{\emptyset}^{D} \qquad \text{implies} \qquad \chi\subseteq D.
			\end{equation*}
		\end{enumerate}
	\end{proposition}
	\begin{proof}
		We will prove the contrapositive statements, i.e.,\ assume $\chi\not\subseteq D$.
		
		Proof of (i): We need to show that $\Lambda_{D}^{\emptyset} \not\geq \Lambda_{\chi}^{\emptyset}$. There exist $V,W\in\A$ with $\dist(V,W)>0$ and a non-empty $\chi'\in\X$, such that 
		\begin{equation*}
			\chi'\subseteq \chi, \qquad \chi'\Subset V, \qquad \text{and} \qquad D \subset W.
		\end{equation*}
		Using Propositions~\ref{prop:main1} and \ref{prop:main12}(i) we have
		\begin{equation*}
			\Lambda_{D}^{\emptyset} \leq \Lambda_{W}^{\emptyset} \quad \text{and} \quad \Lambda_{\chi}^{\emptyset} \geq \Lambda_{\chi'}^{\emptyset}.
		\end{equation*}
		Hence
		\begin{equation*}
				\Lambda_{D}^{\emptyset} - \Lambda_{\chi}^{\emptyset} = (\Lambda_{D}^{\emptyset} - \Lambda_{\emptyset}^{\emptyset}) + (\Lambda_{\emptyset}^{\emptyset} - \Lambda_{\chi}^{\emptyset}) \leq (\Lambda_{W}^{\emptyset} - \Lambda_{\emptyset}^{\emptyset}) + (\Lambda_{\emptyset}^{\emptyset} - \Lambda_{\chi'}^{\emptyset}).
		\end{equation*}
		Now, since $\emptyset \neq \chi' \Subset V$, Proposition~\ref{prop:locpot}(i) gives the existence of a sequence $(f_n)$ in $L^2_\diamond(\Gamma)$ such that
		\begin{equation*}
			\lim_{n\to\infty} \inner{(\Lambda_{D}^{\emptyset} - \Lambda_{\chi}^{\emptyset})f_n,f_n} = -\infty,
		\end{equation*}
		and in particular we conclude that $\Lambda_{D}^{\emptyset} \not\geq \Lambda_{\chi}^{\emptyset}$. 
		
		Proof of (ii): We need to show that $\Lambda_{\emptyset}^{\chi} \not\geq \Lambda_{\emptyset}^{D}$. There exist $V,W\in\A$ with $\dist(V,W)>0$ and a non-empty $\chi'\in\X$, such that 
		\begin{equation*}
			\chi'\subseteq \chi, \qquad \chi'\Subset W, \qquad \text{and} \qquad D \subset V.
		\end{equation*}
		Using Propositions~\ref{prop:main1} and \ref{prop:main12}(ii) we have
		\begin{equation*}
			\Lambda_{\emptyset}^{D} \geq \Lambda_{\emptyset}^{V} \qquad \text{and} \qquad \Lambda_{\emptyset}^{\chi} \leq \Lambda_{\emptyset}^{\chi'}.
		\end{equation*}
		Hence
		\begin{equation*}
			\Lambda_{\emptyset}^{\chi} - \Lambda_{\emptyset}^{D} = (\Lambda_{\emptyset}^{\chi} - \Lambda_{\emptyset}^{\emptyset}) + (\Lambda_{\emptyset}^{\emptyset} - \Lambda_{\emptyset}^{D}) \leq (\Lambda_{\emptyset}^{\chi'} - \Lambda_{\emptyset}^{\emptyset}) + (\Lambda_{\emptyset}^{\emptyset} - \Lambda_{\emptyset}^{V}).
		\end{equation*}
		Now, since $\emptyset \neq \chi' \Subset W$, Proposition~\ref{prop:locpot}(ii) gives the existence of a sequence $(g_n)$ in $L^2_\diamond(\Gamma)$ such that
		\begin{equation*}
			\lim_{n\to\infty} \inner{(\Lambda_{\emptyset}^{\chi} - \Lambda_{\emptyset}^{D})g_n,g_n} = -\infty,
		\end{equation*}
		and in particular we conclude that $\Lambda_{\emptyset}^{\chi} \not\geq \Lambda_{\emptyset}^{D}$. 
	\end{proof}
	
	\subsection*{Acknowledgements}
	
	The research of HG was partially supported by grant 8021--00084B from Independent Research Fund Denmark \textbar\ Natural Sciences. The research of MSV was partially supported by NSF grant DMS-22-05912.
	
	\bibliographystyle{plain}

\end{document}